\documentclass[amstex, dvopdfmx]{amsart}
\usepackage[]{graphicx}
\usepackage{amssymb}
\usepackage{mathrsfs}
\usepackage[all]{xy}
\usepackage{here}

\newtheorem{thm}{Theorem}[section]

\newtheorem{lemma}{Lemma}[section]
\newtheorem{Cor}{Corollary}[section]
\theoremstyle{definition}
\newtheorem{Def}{Definition}[section]
\newtheorem{Ex}{Example}[section]
\theoremstyle{definition}
\newtheorem{Rem}{Remark}[section]

\numberwithin{equation}{section}

\begin{document}

\title[Quasiconformal equivalence]{
    The quasiconformal equivalence of Riemann surfaces and 
    the universal Schottky space}
\author{
    Hiroshige Shiga    
}
\address{Department of Mathematics,
Tokyo Institute of Technology \\
O-okayama, Meguro-ku Tokyo, Japan} 

\email{shiga@math.titech.ac.jp}
\date{\today}    
\thanks{
The author was partially supported by the Ministry of Education, Science, Sports
and Culture, Japan;
Grant-in-Aid for Scientific Research (B), 16H03933}

\begin{abstract}
  In the theory of Teichm\"uller space of Riemann surfaces, we consider the set of Riemann surfaces which are quasiconformally equivalent.
  For topologically finite Riemann surfaces, it is quite easy to examine if they are quasiconformally equivalent or not.
  On the other hand, for Riemann surfaces of topologically infinite type, the situation is rather complicated.
  
  In this paper, after constructing an example which shows the complexity of the problem, we give some geometric conditions for Riemann surfaces to be quasiconformally equivalent.
  
  Our argument enables us to obtain a universal property of the deformation spaces of Schottky regions, which is analogous to the fact that  the universal Teichm\"uller space contains all Teichm\"uller spaces.
  
\end{abstract}
\maketitle
\section{Introduction}
    In the theory of Teichm\"uller space of Riemann surfaces, we consider the set of Riemann surfaces which are quasiconformally equivalent.
    Here, we say that two Riemann surfaces are quasiconformally equivalent if there is a quasiconformal homeomorphism between them.
    Hence, at the first stage of the theory, we have to know a condition for Riemann surfaces to be quasiconformally equivalent.
   
    The condition is quite obvious if the Riemann surfaces are topologically finite. Indeed, the genus, the number of punctures and the number of borders of surfaces are completely determine the quasiconformal equivalence.
    On the other hand, for Riemann surfaces of topologically infinite type, the situation is rather difficult.
    For example, viewing Royden algebras of open Riemann surfaces, Nakai (\cite{Nakai}, see also \cite{Sario-Nakai}) obtains an algebraic criterion for the equivalence.
    He shows that two Riemann surfaces are quasiconformally equivalent if and only if the Royden algebras of those Riemann surfaces are isomorphic.
    However, it is hard to examine the condition in general since the Royden algebras are huge function spaces.
     In this paper, we consider geometric conditions for the quasiconformal equivalence of open Riemann surfaces.

    First, we give examples of Riemann surfaces in order to show the difficulty of the problem.
    We say that two Riemann surfaces $R_1$ and $R_2$ are \emph{quasiconformally equivalent near the ideal boundary} if they are quasiconformally equivalent outside of compact subsets of those surfaces.
    At the first glance, it seems to be true that if two Riemann surfaces are quasiconformally equivalent near the ideal boundary, then they are quasiconformally equivalent. However, it is not true. We may construct a counter example in \S \ref{sec:counter}.
    Namely, we construct two homeomorphic Riemann surfaces $R_1, R_2$ and compact subsets $K_i$ of $R_i$ $(i=1, 2)$ such that $R_1\setminus K_1$ and $R_2\setminus K_2$ are conformally equivalent but $R_1$ and $R_2$ are \emph{not} quasiconformally equivalent.
    This example shows that the quasiconformal equivalence is not a boundary property.
    In the second example, we show that domains given by Schottky groups are not quasiconformally equivalent to domains given by boundary groups of Schottky spaces.

    To give conditions for open Riemann surfaces to be quasiconfomally equivalent, we show a gluing lemma for quasiconformal mappings on Riemann surfaces(Lemma \ref{Lemma:Gluing}).
By using the gluing lemma, we shall give a condition under which Riemann surfaces are quasiconformally equivalent. 
MacManus \cite{MacManus} obtains similar results from a different point of view, that is, a view point of uniform domains, while we are considering the problems from the theory of Riemann surfaces of infinite type.

In \S \ref{sec:universal}, we will discuss a universality of Schottky regions which are complements of the limit sets of Schottky groups.
In fact, we show that Schottky regions are quasiconformally equivalent to each other (Theorem \ref{thm:universalS}).
The result makes a striking contrast to the second example
 in \S \ref{sec:counter}.

At the end, we present \emph{the universal Schottky space} which includes all Schottky spaces.

{\bf Acknowledgment.} The author thanks Prof. H. Fujino for his valuable comments.

\section{Preliminaries}
In this section, we give definitions, terminology and known facts used in the later sections.

Let $R$ be an open Riemann surface.
A sequence $\{W_n\}_{n=1}^{\infty}$ of subdomains of $R$ is called a \emph{regular exhaustion} of $R$ if it satisfies the following conditions.
\begin{enumerate}
	\item Each $W_n$ is a relatively compact domain in $R$ bounded by a finite number of mutually disjoint smooth simple closed curves in $R$;
	\item every connected component of the complement of $W_n$ $(n\in \mathbb N)$ is not compact in $R$;
	\item $W_1\subset W_2\subset \dots \subset W_n\subset W_{n+1}\subset \dots$ and $R=\cup_{n=1}^{\infty}W_n$.
\end{enumerate}

It is known that any open Riemann surface has a regular exhaustion (cf. \cite{Ahlfors-Sario}).

A Riemann surface which is homeomorphic to a triply connected planar domain is called \emph{a pair of pants}.
If a Riemann surface is decomposed into pairs of pants $\{P_n\}$, then we say that the Riemann surface admits a pants decomposition $\{P_n\}$.

\medskip
\noindent
{\bf The Douady-Earle extension}

Let $\phi$ be an orientation preserving homeomorphism from $\mathbb R$ to itself.
The mapping $\phi$ is called \emph{quasi-symmetric} if there exists a constant $M>0$ such that
\begin{equation*}
	M^{-1}\leq \frac{\phi (x)-\phi (x-t)}{\phi (x+t)-\phi (x)}\leq M
\end{equation*}
holds for any $x\in \mathbb R$ and $t>0$.

It is known that(cf. \cite{Ahlfors})  if $\phi : \mathbb R\to \mathbb R$ is quasi-symmetric, then it has a quasiconformal extension to the upper halfplane $\mathbb H$.
Namely, there exists a quasiconformal mapping $f : \mathbb H\to \mathbb H$ whose boundary value on $\mathbb R$ is $\phi$.

In the famous paper by Douady and Earle \cite{Douady-Earle}, they show that every homeomorphism from $\mathbb R$ to itself admits so-called a \emph{conformal natural extension} to $\mathbb H$, which is called the Douady-Earle extension. We denote the Douady-Earle extension of $\phi$ by $E(\phi)$.
The Douady-Earle extension $E(\phi)$ is a homeomorphism on $\mathbb H$ with boundary value $\phi$ and it is conformal natural, that is, for any $\gamma_1, \gamma_2\in \textrm{PSL}(2, \mathbb R)$,
\begin{equation*}
	\gamma_1\circ E(\phi)\circ \gamma_2=E(\gamma_1\circ\phi\circ\gamma_2)
\end{equation*}
holds. Moreover, $E(\phi)$ is real analytic in $\mathbb H$ and if $\phi$ is quasi-symmetric, then $E(\phi)$ is quasiconformal in $\mathbb H$.

\medskip
\noindent
{\bf Teichm\"uller space and Schottky space}

Let $R$ be a hyperbolic Riemann surface and $\Gamma_{R}$ be a Fuchsian group acting on $\mathbb H$ which represents $R$.
A quasiconformal mapping $f : \widehat{\mathbb C}\to \widehat{\mathbb C}$ is called a \emph{quasiconformal deformation} of $\Gamma_{R}$ if it is conformal on the lower halfplane $\mathbb L$ and $f\circ\Gamma_{R}\circ f^{-1}\subset \textrm{PSL}(2, \mathbb C)$.
We say that two quasiconformal deformations $f, g$ of $\Gamma_R$ are equivalent if there exists a M\"obius transformation $A$ such that $g=A\circ f$. 
The Teichm\"uller space $\mathscr{T}(\Gamma_R)$ of the Fuchsian group $\Gamma_R$ is the set of equivalence classes of quasiconformal deformations of $\Gamma_R$.


Let $\textrm{Belt}(\Gamma_R ; \mathbb H)$ be the set of bounded measurable functions $\mu$ on $\mathbb C$ with $\|\mu\|_{\infty}<1$ satisfying
\begin{equation*}
	\mu (\gamma (z))\overline{\gamma'(z)}\gamma'(z)^{-1}=\mu (z) \quad (\textrm{a.e. in }\mathbb H)
\end{equation*}
for any $\gamma\in \Gamma_R$ and $\mu (z)=0$ for any $z\in \mathbb L$.
$\textrm{Belt}(\Gamma_R ; \mathbb H)$ is a complex Banach space by the usual way.

For each $\mu\in \textrm{Belt}(\Gamma_R ; \mathbb H)$, there exists a quasiconformal deformation $w_{\mu} : \widehat{\mathbb C}\to \widehat{\mathbb C}$ of $\Gamma_R$ with
\begin{equation*}
		\frac{\partial w_{\mu}(z)}{\partial \overline z}=\mu (z)\frac{\partial w_{\mu}(z)}{\partial z},\ a.e.
\end{equation*}
Hence, we have a projection $\pi_{T} : \textrm{Belt}(\Gamma_R ; \mathbb H)\to \mathscr{T}(\Gamma_{R})$ by sending $\mu\in \textrm{Belt}(\Gamma_R ; \mathbb H)$ to the equivalence class of $w_{\mu}$.
It is known that the Teichm\"uller space $\mathscr{T}(\Gamma_R)$ admits a complex structure so that the projection $\pi_{T}$ is holomorphic.
It is also known that the complex structures of $\mathscr{T}(\Gamma_R)$ and $\mathscr{T}(\Gamma_{R'})$ are the same if $R$ and $R'$ are quasiconformally equivalent.

If the Riemann surface $R$ is the upper halfplane $\mathbb H$, then the group $\Gamma_R$ is the trivial group $\{id\}$.
We denote by $\mathscr{T}$ the Teichm\"uller space $\mathscr{T}(\{id\})$ and we call it \emph{the universal Teichm\"uller space}.
For any hyperbolic Riemann surface $R$, there exists a natural holomorphic embedding
\begin{equation}
\label{eqn:UniversalT}
	\iota_{R} : \mathscr{T}(\Gamma_{R}) \hookrightarrow \mathscr{T}.
\end{equation}

For more details on Teichm\"uller spaces, see \cite{Hubbard} and \cite{Imayoshi-Taniguchi}.

\medskip

Schottky space is defined in a similar way to Teichm\"uller space.
Let $G_g$ be a Schottky group of genus $g>1$.
A quasiconformal mapping $f : \widehat{\mathbb C}\to \widehat{\mathbb C}$ is called a \emph{quasiconformal deformation} of $G_g$ if $f\circ G_g\circ f^{-1}\subset \textrm{PSL}(2, \mathbb C)$.
We say that two quasiconformal deformations $f, g$ of $G_g$ are equivalent if there exists a M\"obius transformation $A$ such that $g$ is homotopic to $A\circ f$ $\mathrm{rel }$ $\Lambda({G_g})$.
The Schottky space $\mathscr{S}_g$ of genus $g$ is the set of equivalence classes of quasiconformal deformations of $G_g$.

Let $\textrm{Belt}(G_g; \mathbb C)$ be the set of bounded measurable functions $\mu$ on $\mathbb C$ with $\|\mu\|_{\infty}<1$ satisfying
\begin{equation*}
	\mu (\gamma (z))\overline{\gamma'(z)}\gamma'(z)^{-1}=\mu (z),\ a.e.
\end{equation*}
for any $\gamma\in G_g$.
By the same way as in Teichm\"uller spaces, we have a projection $\pi_{S} : \textrm{Belt}(S_g; \mathbb C)\to \mathscr{S}_g$ and the Schottky space $\mathscr{S}_g$ admits a complex structure so that the projection $\pi_S$ is holomorphic.
It is known that the complex structure of $\mathscr{S}_g$ depends only on the genus $g$.

\begin{Rem}
	The Schottky space defined above is called the strong deformation space of $G_g$ in \cite{Kra}, in which the complex structure of the space is discussed.
\end{Rem}

\medskip
\noindent
{\bf Teichm\"uller space of a closed set.}

Let $E$ be a closed set in $\widehat{\mathbb C}$.
We denote by $\mathrm{Belt}(\mathbb C)=\mathrm{Belt}(\{id\}; \mathbb C)$ the set of bounded measurable functions $\mu$ on $\mathbb C$ with $\|\mu\|_{\infty}<1$.
Two functions $\mu_1, \mu_2$ are said to be equivalent if there exists a M\"obius transformation $A$ such that $A\circ w_{\mu_1}$ is homotopic to $w_{\mu_2}$ rel $E$.
We define Teichm\"uller space of $E$, which is denoted by $\mathscr{T}(E)$, is defined by the set of equivalence classes.

\section{Examples of Riemann surfaces on quasiconformal non-equivalence}
\label{sec:counter}
In this section, we construct two examples of pairs of Riemann surfaces which are not quasiconformally equivalent.
In the first example, we construct two Riemann surfaces $R_1$ and $R_2$ which are quasiconformally equivalent near the ideal boundary but not quasiconformally equivalent.
The second one is an example of Riemann surfaces defined by Cantor sets. The example has an own interest itself and it is also related to the result in Theorem \ref{thm:universalS} in \S 6.

\begin{Ex}
Put $a_n=(n!)^{-1}$and take pairs of pants $P_n$ bounded by three hyperbolic closed geodesics whose length are $1, 1$ and $a_n$ $(n=0, 1, 2, \dots)$. 	
We glue $P_{n}$ and $P_{n+1}$ along two boundary curves with length $1$ to make a Riemann surface $T_n$ of genus $1$ with two boundary curves of lengths $a_n$ and $a_{n+1}$.
Since $T_n$ and $T_{n+1}$ have a boundary curve of the length $a_{n+1}$, we may glue them along the boundary curves.
By repeating this operation for $n=0, 1, 2, \dots$, we get a Riemann surface $R_{1}'=\bigcup_{n=0}^{\infty}\overline{T_n}$ which is a Riemann surface of infinite genus with a geodesic boundary curve of length $1$.
We take a Riemann surface $S$ of genus $1$ with a geodesic boundary curve of length $1$ by gluing two boundary curves of $P_0$.
Gluing $R_{1}'$ and $S$ along the boundary curves, we have an open Riemann surface $R_1$ of infinite genus.

Next, we make a Riemann surface $R_{2}'$ by the same way as $R_{1}'$ but we do it from $n=1$ instead of $n=0$ for $R_{1}'$.
Then, $R_{2}'$ is still a Riemann surface of infinite genus with a geodesic boundary of length $1$.
Hence, we can glue $R_{2}'$ and $S$ along the boundary curves, we have an open Riemann surface $R_2$ of infinite genus (\textsc{Figure} \ref{FigCounter}).

\begin{figure}
\centering
	\includegraphics[width=14cm]{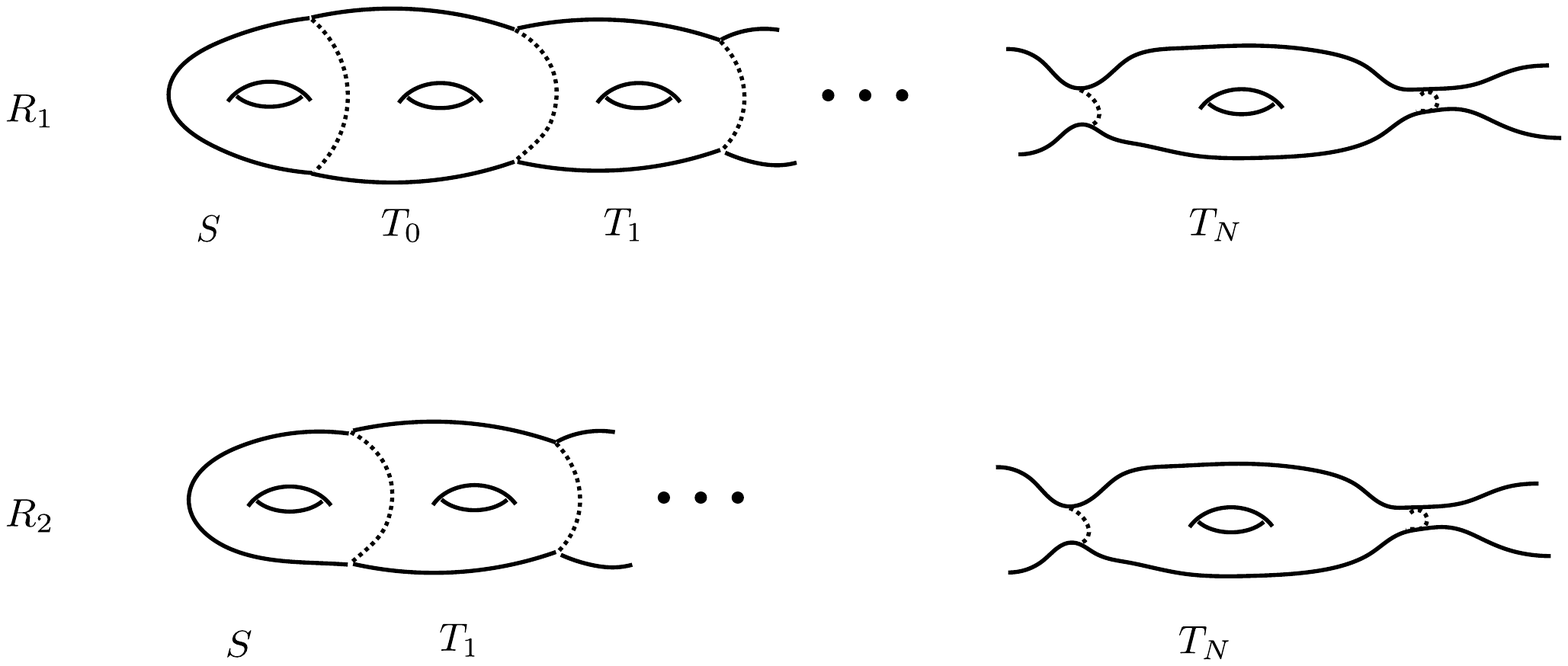}
	\caption{}
	\label{FigCounter}	
\end{figure}

Obviously, both $R_1$ and $R_2$ are homeomorphic and they have the same subsurface $\bigcup_{n=1}^{\infty}\overline{T_n}$. 
Hence, $R_1\setminus K_1$ and $R_2\setminus K_2$ are conformally equivalent for $K_1=\overline{S\cup T_0}$ and $K_2=\overline{S}$.
In particular, they are quasiconformally equivalent near the ideal boundary.
However, we may show that there are no quasiconformal mappings between $R_1$ and $R_2$.

Suppose that there exists a $K$-quasiconformal mapping $F : R_1\to R_2$ for some $K\geq 1$. 
We take a sufficiently large $N\in \mathbb N$ with $N>K$.
We consider the closed geodesic $\alpha_{N}$ of $\partial T_{N}\subset R_1$ with length $a_N$ and the geodesic $[F(\alpha_N)]$ homotopic to $F(\alpha_N)$ in $R_2$.
It follows from Wolpert's formula (\cite{ShigaHyper},\cite{Wolpert}) that the hyperbolic length $\ell ([F(\alpha_N)])$ of $[F(\alpha_N)]$ in $R_2$ satisfies an inequality,
\begin{equation*}
	K^{-1}a_N\leq \ell([F(\alpha_N)])\leq Ka_N.
\end{equation*}
Hence, we have
\begin{equation}
\label{eqn:Wolpert}
	a_{N+1}=\frac{1}{(N+1)!}<N^{-1}a_N\leq \ell([F(\alpha_N)])\leq Na_N<\frac{N}{N!}=a_{N-1}.
\end{equation}

If the geodesic $[F(\alpha_N)]$ transversely intersects with some $\alpha_i$ in $R_2$, then it follows from the collar theorem (cf. \cite{Buser}) that the length $\ell ([F(\alpha_N)])$ is large enough.
If $[F(\alpha_N)]\cap \alpha_i=\emptyset$ for any $i\in \mathbb N$, from the geometry of $S$ and $T_n$ $(n\in \mathbb N)$ we see that $\ell ([F(\alpha_N)])$ is larger than $a_N$ for a sufficiently large $N$.

Hence, we conclude that only the closed geodesic of $\overline{T_N}\cap\overline{T_{N+1}}$ in $R_2$ has the length satisfying (\ref{eqn:Wolpert}).
Therefore, the subsurface $\overline{S}\cup \bigcup_{n=0}^{N-1}\overline{T_n}$ of $R_1$ which is of genus $N+1$ has to be mapped a subsurface of $R_2$ of genus $N$.
It is absurd because $F$ is a homeomorphism. Thus, we have a contradiction.
\end{Ex}

\begin{Ex}
\label{Ex:nonQC}
	Let $G$ be a Schottky group of genus $g>1$. 
	The group is constructed from $2g$ (topological) closed disks $D_1, D_2, \dots , D_{2g}$ with $D_i\cap D_j=\emptyset$ $(i\not=j)$ and $\gamma_i\in \textrm{PSL}(2, \mathbb C)$ $(i=1, 2, \dots , g)$ which map the outside of $D_{2i-1}$ onto the inside of $D_{2i}$.
	The group $G$ is a Kleinian group generated by $\gamma_1, \gamma_2. \dots , \gamma_g$ and it is a purely loxodromic free group of rank $g$.
	The region of discontinuity $\Omega(G)$ of $G$ is a connected domain in $\widehat{\mathbb C}$ and the complement $\Lambda(G)$, the limit set of $G$, is a Cantor set. Thus, $\Omega(G)$ is an open Riemann surface of infinite type.
	
	Now, we consider a Kleinian group $G'$ of \emph{Schottky type} with cusps.
	We construct the group $G'$ as follows.
	
	Take $2g$ closed disks $D'_1, D'_2, \dots , D'_{2g}$ such as $D'_i\cap D'_j=\emptyset$ for $1\leq i<j\leq 2g-1$, $D'_i\cap D'_{2g}=\emptyset$ for $1\leq i\leq 2g-2$ but $D'_{2g}$ is tangential to $D'_{2g-1}$ at one point $z_0$. 
	We also take $\delta_i\in \textrm{PSL}(2, \mathbb C)$ $(i=1, 2, \dots g-1)$ which map the outside of $D'_{2i-1}$ onto the inside of $D'_{2i}$, and $\delta_{g}\in \textrm{PSL}(2, \mathbb C)$ which maps the outside of $D'_{2g-1}$ onto the inside of $D'_{2g}$ fixing  $z_0$.
	Hence, $\delta_g$ is a parabolic transformation with the fixed point $z_0$.	
	The group $G'$ is generated by $\delta_1, \delta_2, \dots , \delta_{g}$.
	The group $G'$ is still a Kleinian group and a free group of rank $g$, but it contains parabolic elements $\delta_{g}$.
	
	We may take a sequence $\{G_n\}_{n=1}^{\infty}$ of Schottky groups of genus $g$ such that it converges to $G'$. 
	Hence, the group $G'$ is regarded as a group on the boundary of Schottky space.

	The limit set $\Lambda(G')$ of $G'$ is also a Cantor set and the region of discontinuity $\Omega(G')$ is an open Riemann surface of infinite type. 
	
	Thus, we have two open Riemann surfaces $\Omega(G)$ and $\Omega(G')$ of infinite type both of which are complements of some Cantor sets.
	Then, we insist the following.
	
	\medskip
	
	\noindent
	{\bf Claim:} $\Omega(G)$ and $\Omega(G')$ are not quasiconformally equivalent.
	
	\medskip
	
	Since both $G$ and $G'$ are quasiconformal deformations of  Fuchsian groups, we may assume that $G$ and $G'$ are Fuchsian groups, so that $\Lambda (G), \Lambda(G') \subset \mathbb R$.
	Suppose that there exists a quasiconformal mapping $f$ from $\Omega (G)$ onto $\Omega(G')$.
	Then we have known the following (\cite{ShigaKlein} Theorem 1. 2 and Corollary 1. 3).
	\begin{enumerate}
		\item the mapping $f$ is extended to a quasiconformal mapping from $\widehat{\mathbb C}$ onto itself. We use the same letter $f$ for the extended mapping;
		\item the mapping $f$ is extended to a homeomorphism of the Martin compactifications. We denote the extended homeomorphism by $f^{*}$ (as for the Martin compactification, see \cite{Constantinescu-Cornea}).
	\end{enumerate}
	
	Let $p\in \Lambda (G')$ be a parabolic fixed point. From (1) above, there exists a point $q\in \Lambda (G)$ such that $f(q)=p$.
	Moreover, it follows from (2) that there exists a unique limit of $f^{*}(z)$ as $z\to q$ in the Martin compactification of $\Omega (G)$.
	On the other hand, in the Martin compactification of $\Omega (G')$, there are more than two points over a parabolic fixed point (\cite{ShigaKlein} Theorem 1.\ 1 (A), see also \cite{Segawa}).
	Therefore, we may find a non-convergent sequence $\{f^{*}(z_n)\}_{n=1}^{\infty}$ as $z_n\to q$. 
	Thus, we have a contradiction. 

\end{Ex}

\section{A gluing lemma}
In this section, we shall prove the following lemma.
\begin{lemma}
\label{Lemma:Gluing}
	Let $X, Y$ be Riemann surfaces. 
	We consider simple closed curves $\alpha\subset X$ and $\beta\subset Y$ with $X\setminus{\alpha}=X_{1}\sqcup X_2$ and $Y\setminus\beta =Y_1\sqcup Y_2$, respectively.
	Suppose that there exist quasiconformal mappings $f_{i} : X_{i}\to Y_{i}$ $(i=1, 2)$ such that $f_1(\alpha)=f_2(\alpha)=\beta$.
	Then, there exists a quasiconformal mapping $f : X\to Y$. Moreover, the maximal dilatation of $f$ depends only on those of $f_1, f_2$ and the local behavior of those mappings near $\alpha$.
\end{lemma}
\begin{Rem}
Since $\alpha$ is a simple closed curve, the quasiconformal mappings $f_1$ and $f_2$ are extended homeomorphically to $\alpha$.
We use the fact in the statement of the above lemma.
\end{Rem}
\begin{Rem}
	If we suppose that $\alpha$ is piecewise smooth and $f_1, f_2$ agree on $\alpha$, then the conclusion is easy.
	But we do not assume them in this lemma.
\end{Rem}
\begin{proof}
	We take simple closed curves $\alpha_{i}\subset X_i$ $(i=1, 2)$ near $\alpha$ so that $\alpha$ and $\alpha_i$ bound annuli $A_i\subset X_j$.
	We put $B_i=f_i (A_i)$ and $\beta_i=f_i(\alpha_i)$.
	Then, $B_i$ are also annuli, which are bounded by $\beta$ and $\beta_i$ $(i=1, 2)$.
	First of all, we show that $f_1$ and $f_2$ can be real analytic on $\alpha_1$ and $\alpha_2$, respectively.
	
		There exist $r_i, k_i>1$ such that each $A_i$ is conformally equivalent to a circular annulus 
	$$
	\mathcal{A}_{i}:=\{z\in \mathbb C \mid 1<|z|<r_i\}\simeq \mathbb H/<z\mapsto k_iz>
	$$
	 via a conformal mapping $\varphi_i :\mathcal{A}_i\to A_i$ $(i=1, 2)$. 
	 We also take $\rho_i, \kappa_i >1$ so that each $B_i$ is conformally equivalent to a circular annulus
	 \begin{equation*}
	 	\mathcal{B}_i=\{z\in \mathbb C\mid 1<|z|<\rho_i\}\simeq \mathbb{H}/<z\mapsto \kappa_iz >\}
	 \end{equation*}
	 via $\psi_i :\mathcal{B}_i\to B_i$.
	 
	 Then,  $\phi_i :=\psi_i^{-1}\circ f_i|_{A_i}\circ\varphi_i$ from $\mathcal{A}_i$ onto $\mathcal{B}_i$ are lifted to  quasiconformal mappings $\widehat{\phi}_i :\mathbb H\to \mathbb H$ with
	 \begin{equation*}
	 	\widehat{\phi}_i(k_iz)=\kappa_i\widehat{\phi}_i(z)
	 \end{equation*}
	 for any $z\in \overline{\mathbb H}$. In particular,
	 \begin{equation}
	 \label{eqn:realEuiv}
	 	\widehat{\phi}_i(k_ix)=\kappa_i\widehat{\phi}_i(x)
	 \end{equation}
	 holds for any $x\in \mathbb R$.
	 
	 We take the Douady-Earle extension $\widehat \Phi_i$ of $\widehat{\phi}_i|_{\mathbb R}$.
	 Since $\widehat{\phi}_i$ satisfy (\ref{eqn:realEuiv}) on $\mathbb{R}$, $\widehat \Phi_i$ also satisfy the equations on $\overline{\mathbb H}$. 
	 Moreover, they are real analytic in $\mathbb H$.
	 Therefore, the quasiconformal mappings $\widehat \Phi_i : \mathbb H\to \mathbb H$ are projected  quasiconformal mappings $\Phi_i : \mathcal{A}_i\to \mathcal{B}_i$.
	 Hence, $F_i :=\psi_i\circ \Phi_i\circ\varphi_i^{-1} : A_i\to B_i$ are  real analytic quasiconformal mappings with the same boundary values as $f_i|_{A_i}$ $(i=1,2)$.
	 
	 We define quasiconformal mappings $\tilde{f_i}$ from $X_i$ onto $Y_i$
	 by $f_i$ on $X_i\setminus A_i$ and $F_i$ on $A_i\cup\{\alpha_i\}$. They are real analytic in $A_i$.
	 Let $\tilde{\alpha}_i$ be non-trivial smooth Jordan curves in $A_i$. 
	 Then, $\tilde f_i$ are real analytic on $\tilde\alpha_i$.
	 Thus, by considering $\tilde f_i$ and $\tilde{\alpha}_i$ instead of $f_i$ and $\alpha_i$, respectively, we may assume that $f_i$ are real analytic on $\alpha_i$.
	 
	 Now, we consider an annulus $A$ in $X$ bounded by $\alpha_1$ and $\alpha_2$.
	 We also consider an annulus $B$ in $Y$ bounded by $\beta_1 :=f_1(\alpha)$ and $\beta_2 :=f_2(\alpha)$. 
	 We take $r, k>1$ so that $A$ is conformally equivalent to the circular annulus
	 \begin{equation*}
	 	\mathcal{A}:=\{z\in \mathbb C\mid 1<|z|<r \}\simeq \mathbb H/<z\mapsto kz>
	 \end{equation*}
	 via a conformal mapping $h_A : \mathcal{A}\to A$.
	 We also take $\rho, \kappa >1$ so that $B$ is conformally equivalent to the circular annulus
	 \begin{equation*}
	 	\mathcal{B}:=\{z\in \mathbb C\mid 1<|z|<\rho\}\simeq \mathbb{H}/<z\mapsto kz>
	 \end{equation*}
	 via a conformal mapping $h_B :\mathcal B\to B$.
	 
	 We denote by $\pi_A : \overline{\mathbb H}\setminus \{0, \infty\}\to \overline {A} \simeq \overline{\mathbb H}\setminus\{0, \infty\}/<z\mapsto kz>$ and $\pi_B : \overline{\mathbb H}\setminus \{0, \infty\}\to \overline {B} \simeq \overline{\mathbb H}\setminus\{0, \infty\}/<z\mapsto \kappa z>$, the quotient mappings for $\overline{A}$ and $\overline{B}$, respectively.
	 We may assume that $\pi_A (\mathbb R_{<0})=\alpha_1$, $\pi_A (\mathbb R_{>0})=\alpha_2$, $\pi_B (\mathbb R_{<0})=\beta_1$ and $\pi_B (\mathbb R_{>0})=\beta_2$.
	 Then, the smooth homeomorphism $f_1|_{\alpha_1} : \alpha_1 \to \beta_1$ is lifted to a smooth homeomorphism from $\mathbb R_{<0}$ to itself and $f_2|_{\alpha_2} : \alpha_2 \to \beta_2$ is also lifted to a smooth homeomorphism from $\mathbb R_{>0}$ to itself.
	 Thus, we have a strictly increasing homeomorphism $\Psi$ on $\mathbb R$ onto itself which are smooth in $\mathbb R\setminus\{0\}$ with $\Psi (0)=0$.
	 The mapping $\Psi$ satisfies
	 \begin{equation}
	 \label{eqn:realEquiv2}
	 	\Psi (kx)=\kappa \Psi (x)
	 \end{equation}
	 for any $x\in \mathbb R$.
	 
	 We may normalize the function as $\Psi (1)=1$ and $\Psi (-1)=-1$.
	 We show that $\Psi$ is quasi-symmetric on $\mathbb R$.
	 
	 We put
	 \begin{equation*}
	 	M=\sup_{x>0, t>0}\frac{\Psi (x)-\Psi(x-t)}{\Psi (x+t)-\Psi (x)}
	 \end{equation*}
	 and
	 \begin{equation*}
	 	m=\inf_{x>0. t>0}\frac{\Psi (x)-\Psi(x-t)}{\Psi (x+t)-\Psi (x)}.
	 \end{equation*}
	 We show that $0<m\leq M<\infty$ in several steps.
	 
	 If $x=k y$ and $t=k s$ $(s>1)$, then we have from (\ref{eqn:realEquiv2})
	 \begin{eqnarray*}
	 	\Psi (x)-\Psi (x-t)&=&\Psi (k y)-\Psi (k (y-s))=\kappa (\Psi (y)-\Psi (y-s)),\\
	 	\Psi (x+t)-\Psi (x)&=&\kappa (\Psi (y+s)-\Psi (y)).	 	
	 \end{eqnarray*}
	 Thus, we see
	 \begin{equation}
	 	M=\sup_{x\in\{0\}\cup [1, k], t>0}\frac{\Psi (x)-\Psi(x-t)}{\Psi (x+t)-\Psi (x)}
	 \end{equation}
	 and
	 \begin{equation}
	 		 	m=\inf_{x\in\{0\}\cup [1, k], t>0}\frac{\Psi (x)-\Psi(x-t)}{\Psi (x+t)-\Psi (x)}.
	 \end{equation}
	 \noindent
	 (i) If $x\in [1, k]$ and $0<t\leq \frac{1}{2}$, then we have
	 \begin{equation*}
	 	\Psi (x)-\Psi (x-t)=\Psi'(x-\theta t)t
	 \end{equation*}
	 and
	 \begin{equation*}
	 	\Psi (x+t)-\Psi (x)=\Psi'(x+\theta' t)t
	 \end{equation*}
	 for some $\theta, \theta'\in (0, 1)$.
	 Thus,
	 \begin{equation*}
	 	\frac{\Psi (x)-\Psi(x-t)}{\Psi (x+t)-\Psi (x)}=\frac{\Psi'(x-\theta t)}{\Psi' (x+\theta' t)}.
	 \end{equation*}
	 Since $x\in [1, k]$, 
	 \begin{equation*}
	 	\frac{1}{2}\leq x-\theta t<x+\theta' t\leq k+\frac{1}{2}.
	 \end{equation*}
	 We conclude that there exist $0<m_1<M_1<\infty$ such that
	 \begin{equation}
	 	m_1\leq \frac{\Psi (x)-\Psi(x-t)}{\Psi (x+t)-\Psi (x)}\leq M_1
	 \end{equation}
	 for any $x\in [1, k]$ and $t\in (0, \frac{1}{2}]$.
	 
	 \medskip
	 
	 \noindent
	 (ii) If $\frac{1}{2}<t<x$, we have
	 \begin{equation*}
	 	\Psi (x)-\Psi (x-t)\leq \Psi (x)\leq \Psi (k)=\kappa
	 \end{equation*}
	 and
	 \begin{equation*}
	 	\Psi (x+t)-\Psi (x)\geq \Psi \left (x+\frac{1}{2}\right )-\Psi (x).
	 \end{equation*}
	 For $\widetilde{m}_2=\inf_{1\leq x\leq k}\left\{\Psi \left (x+\frac{1}{2}\right )-\Psi (x)\right\}>0$, we get
	 \begin{equation*}
	 	\frac{\Psi (x)-\Psi(x-t)}{\Psi (x+t)-\Psi (x)}\leq\frac{\kappa}{\widetilde{m}_2} <\infty .
	 \end{equation*}
	 Also, we have
	 \begin{equation*}
	 	\Psi (x)-\Psi (x-t)\geq \Psi (x)-\Psi \left (x-\frac{1}{2}\right )
	 \end{equation*}
	 and
	 \begin{equation*}
	 	\Psi (x+t)-\Psi (x)\leq \Psi (x+t)\leq \Psi (2x)\leq \Psi (2k)=\kappa \Psi (2).
	 \end{equation*}
	 For $\widehat{m}_2=\inf_{1\leq x\leq k}\left\{\Psi (x)-\Psi \left (x-\frac{1}{2}\right )\right\}>0$, we get
	 \begin{equation*}
	 	\frac{\Psi (x)-\Psi(x-t)}{\Psi (x+t)-\Psi (x)}\geq \frac{\widehat{m}_2}{\kappa\Psi (2)}>0.
	 \end{equation*}
	 \medskip
	 
	 \noindent
	 (iii) If $x\in [1, k]$ and $t\geq \frac{1}{2}$, then we put
	 	 \begin{equation*}
	 	M_3=\sup_{1\leq x\leq k, t\geq \frac{1}{2}}\frac{\Psi (x)-\Psi(x-t)}{\Psi (x+t)-\Psi (x)}
	 \end{equation*}
	 and
	 \begin{equation*}
	 	m_3=\inf_{1\leq x\leq k, t\geq \frac{1}{2}}\frac{\Psi (x)-\Psi(x-t)}{\Psi (x+t)-\Psi (x)}.
	 \end{equation*}
	 We take sequences $\{x_n\}, \{t_n\}$ so that $x_n\in [1, k], t_n\geq \frac{1}{2}$ and
	 \begin{equation*}
	 	\lim_{n\to\infty}\frac{\Psi (x_n)-\Psi(x_n-t_n)}{\Psi (x_n+t_n)-\Psi (x_n)}=M_3.
	 \end{equation*}
	 If $\{t_n\}$ is bounded, it is obvious that $M_3<\infty$.
	 We suppose that $\{t_n\}$ is unbounded.
	 Since $x_n\in [1, k]$, we have
	 \begin{equation*}
	 	x_n - t_n\in [1-t_n, k-t_n], \quad x_n+t_n\in [1+t_n, k+t_n].
	 \end{equation*}
	 Hence, we have
	 \begin{equation*}
	 	\Psi (x_n)-\Psi (x_n-t_n)\leq \Psi (k)-\Psi (1-t_n)=\kappa-\Psi(1-t_n) ,
	 \end{equation*}
	 and
	 \begin{equation*}
	 	\Psi (x_n+t_n)-\Psi (x_n)\geq \Psi (1+t_n)-\Psi (k)=\Psi(1+t_n)-\kappa .
	 \end{equation*}
	  We take $m(n)\in \mathbb N$ such that
	 \begin{equation*}
	 	k^{m(n)}\leq t_n\leq k^{m(n)+1}.
	 \end{equation*}
	 Note that $m(n)\to \infty$ as $n\to \infty$. Then
	 \begin{equation*}
	 	\Psi (1-t_n)\geq \Psi (-t_n)\geq \Psi (-k^{m(n)+1})=\kappa^{m(n)+1}\Psi (-1)=-\kappa^{m(n)+1},
	 \end{equation*}
	 	and
	 \begin{equation*}
	 	\Psi (1+t_n)\geq \Psi (k^{m(n)})=\kappa^{m(n)}.
	 \end{equation*}
	 Thus, we have
\begin{equation*}
	\frac{\Psi (x_n)-\Psi(x_n-t_n)}{\Psi (x_n+t_n)-\Psi (x_n)}\leq \frac{1+\kappa^{m(n)}}{\kappa^{m(n)-1}-1},
\end{equation*}
	and we get
 \begin{equation*}
 	M_3\leq  \kappa,
 \end{equation*} 
 as $m(n)\to \infty$.	
 A similar argument shows that $m_3 >0$.
 
 Thus, we conclude that $0<m<M<\infty$.
 By using the same argument as above, we can show that
 \begin{equation*}
 	0<\inf_{x<0, 0<t}\frac{\Psi (x)-\Psi(x-t)}{\Psi (x+t)-\Psi (x)}\leq \sup_{x<0, 0<t}\frac{\Psi (x)-\Psi(x-t)}{\Psi (x+t)-\Psi (x)}<\infty.
 \end{equation*}
 
 \medskip
 
 \noindent
 (iv) If $x=0$, we have
\begin{eqnarray*}
	 	\kappa^{n-1} &\leq \Psi (t)& \leq \kappa^{n} \\
-\kappa^{n}=\kappa^{n}\Psi (-1)&\leq \Psi (-t)&\leq \kappa^{n-1}\Psi(-1)=-\kappa^{n-1}
\end{eqnarray*}
if $k^{n-1}\leq t\leq k^{n}$ $(n\in \mathbb N)$. Hence,
\begin{equation*}
	\kappa^{-1}\leq \frac{\Psi (0)-\Psi (-t)}{\Psi (t)-\Psi (0)}\leq \kappa .
\end{equation*}
The same argument gives us the same estimate for $\kappa^{-n}\leq t\leq \kappa^{-n+1}$ $(n\in \mathbb N)$.

It follows from (i) -- (iv) that $\Psi$ is quasi-symmetric on $\mathbb R$.

\medskip

Now, we take the Douday-Earle extension $E({\Psi})$ of $\Psi$. 
It is a quasiconformal self-mapping of $\mathbb H$ because of the quasi-symmetricity of $\Psi$.
Since $\Psi$ satisfies (\ref{eqn:realEquiv2}), the equation
\begin{equation*}
	E({\Psi})(k z)=\kappa E({\Psi})(z)
\end{equation*}
also holds for any $z\in \mathbb H$.
Therefore, $E({\Psi})$ is projected to a quasiconformal mapping $\psi$ from $A$ to $B$. 
Moreover, we have $\psi|_{\alpha_1}=f_1|_{\alpha_1}$, $\psi|_{\alpha_2}=f_2|_{\alpha_2}$.
We define a map $f : X\to Y$ by
\begin{equation*}
	f(p)=
	\begin{cases}
	f_i(p)\quad (p\in X_i\setminus A; i=1, 2) \\
	\psi (p)\quad (p\in \overline{A})	.
	\end{cases}
\end{equation*}
The map $f$ is a homeomorphism and quasiconformal except on $\alpha_1\cup\alpha_2$. 
 It follows from the removability for quasiconformal mapping that $f$ is quasiconformal on $X$.
 Moreover, from the construction we see that the maximal dilatation of $f$ depends only on those of $f_i$ and the local behavior of them near $\alpha$.	
 \end{proof}
 
 \section{Conditions for the quasiconformal equivalence \\ of Riemann surfaces}
 Let $R_1, R_2$ be open Riemann surfaces which are homeomorphic to each other.
 Suppose that $R_1$ and $R_2$ are quasiconformally equivalent {near the ideal boundaries}, namely,
 there exist compact subsets $K_j$ of $R_j$ $(j=1, 2)$ and a quasiconformal mapping such that $f(R_1\setminus K_1)=R_2\setminus K_2$.
 As we have seen in the previous section, the quasiconformal equivalence near the ideal boundaries does not imply the quasiconformal equivalence of the surfaces in general.
 In this section, we will give sufficient conditions for two open Riemann surfaces which are  quasiconformally equivalent near the ideal boundaries to be quasiconformally equivalent.
 
 We say that an open Riemann surface $R$ admits a \emph{bounded pants decomposition} if there exists a pants decomposition $\{P_n\}_{n=1}^{\infty}$ of $R$ such that each $P_n$ is bounded by hyperbolic closed geodesics and the lengths of the geodesics are in $[M^{-1}, M]$, where $M>0$ is a constant independent of $n$.
 
 \begin{Def}
 	Let $E$ be an end of an open Riemann surface $R$.
 	We say that $E$ is an {\it infinite ladder end} (ILE) if $E$ is an end of infinite genus having a bounded pants decomposition $\{P_n\}_{n=1}^{\infty}$ given by the dotted lines as in Figure \ref{FigILE}.
 \end{Def}
 
 \begin{figure}
\centering
\includegraphics[width=14cm]{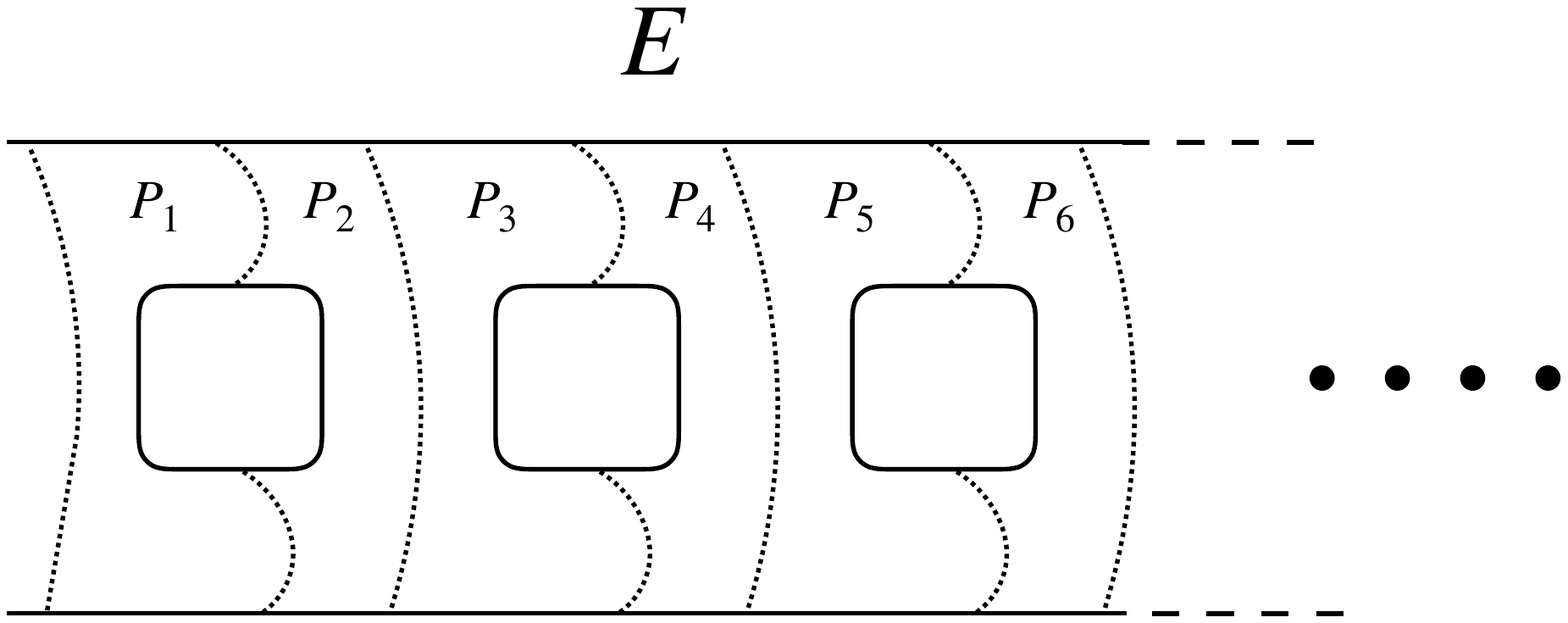}
\caption{}
\label{FigILE}
\end{figure}
 
 \begin{thm}
 	Let $R_1, R_2$ be homeomorphic open Riemann surfaces which are quasiconformally equivalent near the ideal boundaries.
 	\begin{enumerate}
 		\item If the genus of $R_1$ is finite, then $R_1$ and $R_2$ are quasiconformal equivalent.
 		\item If $R_1$ has an ILE, then $R_1$ and $R_2$ are quasiconformally equivalent.
 	\end{enumerate}
 	
 \end{thm}
 \begin{proof}
 From the assumption, there exist compact subsets $K_i$ of $R_i$ $(i=1, 2)$ and a quasiconformal mapping $f$ on $R_1\setminus K_1$ such that $f(R_1\setminus K_1)=R_2\setminus K_2$.
 
 \medskip
 	(1)
 	Let $R_1=\cup_{n=1}^{\infty}W_n$ be a regular exhaustion of $R_1$. 
 	Each $W_n$ is a relatively compact subregion of $R_1$ bounded by a finite number of mutually disjoint simple closed curves, and every connected component of the complement of $W_n$ is not relatively compact in $R_1$. 
 	Hence, there exists $N\in \mathbb N$ such that $K_1\subset W_N$ and the genus of $W_N$ is the same as that of $R_1$.
 	Also, the number of the connected components of $R_1\setminus W_N$ is not more than that of the boundary components of $W_N$.
 	Thus, it has to be finite.
 	
 	Let $E_1, \dots , E_k$ be the set of connected components of $R_1\setminus W_N$.
 	Since $W_N$ is of the same genus as $R_1$, every $E_j$ is a planar and so is $f(E_j)$.
 	Hence, we may take a simple closed curve $\alpha_j$ in $E_j$ which separates the ideal boundary of $E_j$ and the relative boundaries of $E_j$.
 	We see that there is a unique connected component of $R_2\setminus\cup_{j=1}^{k}f(\alpha_j)$ which is relatively compact in $R_2$.
 	
 	Indeed, if there are two relatively compact connected components in $R_2\setminus\cup_{j=1}^{k}f(\alpha_j)$, then each of them together with its connected components of the complement is a subdomain of $R_2$ with no relative boundaries.
 	It is absurd because of the connectivity of $R_2$.
 	It has to be unique.
 	
 	We denote by $S_2$ the relatively compact connected component of $R_2\setminus\cup_{j=1}^{k}f(\alpha_j)$.
 	It is also seen that there is a unique connected component of $R_1\setminus\cup_{j=1}^{k}\alpha_j$.
 	The component is denoted by $S_1$.
 	Then, both $S_1$ and $S_2$ are open Riemann surfaces of the same genus bounded by the same number of simple closed curves.
 	Hence, they are quasiconformally equivalent as well as their complements.
 	Thus, we see from Lemma \ref{Lemma:Gluing} that $R_1$ and $R_2$ are quasiconformally equivalent.
 	
 \medskip
 	(2) 
 	Let $E\subset R_1$ be an ILE of $R_1$ with a bounded pants decomposition $\{P_n\}_{n=1}^{\infty}$ as \textsc{Figure} \ref{FigILE} shows. 
 	Every boundary curve of $P_n$ $(n\in \mathbb N)$ is the hyperbolic geodesic whose length is in $[M^{-1}, M]$ for some $M>0$ independent of $n$.
 	
 	From the assumption, there exist compact subset $K_i$ of $R_i$ $(i=1, 2)$ and a quasiconformal mapping $f : R_1\setminus K_1\to R_2\setminus K_2$.
 	We may assume that $K_1$ is the closure of a regular region $S_1$ of $R_1$ and $E$ is a connected component of $R_1\setminus S_1$.
 	We put $S_2=R_2\setminus f(R_1\setminus S_1)$.
 	
 	Since $K_1=\overline{S_1}$, the boundary $\partial K_1=\partial S_1$ consists of finitely many Jordan curves in $R_1$.
 	Hence, so is $f(\partial K_1)=\partial S_2$.
 	In particular, the number of boundary components of $S_2$ are the same as that of $S_1$.
 	If the genus of $S_2$ is the same as that of $S_1$, then $S_1$ and $S_2$ are quasiconformally equivalent.
 	Thus, it follows from Lemma \ref{Lemma:Gluing} that $R_1$ and $R_2$ are quasiconformally equivalent.
 	
 	Suppose that the genus of $S_2$ is greater than the genus of $S_1$ and let $m\in \mathbb N$ be the difference of them.
 	For  a bounded pants decomposition $\{P_n\}_{n=1}^{\infty}$ of $E$ as \textsc{Figure} \ref{FigILE}, pairs of pants $P_1, \dots , P_{2m}$ makes a regular region $W_m$ of genus $m$ with two boundary components.
 	By gluing $S_1$ and $W_m$, we get a regular region $S_1'$ of the same genus as that of $S_2$.
 	We also see that $S_1'$ is bounded by the same number of closed curves as $S_2$.
 	Therefore, $S_1'$ and $S_2$ are quasiconformally equivalent.
 	
 	Now, we consider an end $E_m := E\setminus \cup_{n=1}^{2m}\overline{P_n}$.
 	The end $E_m$ is still an ILE end with a bounded pants decomposition $\{P_n\}_{n\geq m+1}$.
 	On the other hand, the end $E':=f(E)$ is also an ILE and it admits a bounded pants decomposition $\{P'_n\}_{n=1}^{\infty}$ as \textsc{Figure} \ref{FigILE}.
 	It follow from Wolpert's formula that the hyperbolic length of any boundary curve of $P'_n$ is in $[K(f)^{-1}M^{-1}, K(f)M]$, where $K(f)$ is the maximal dilatation of $f$.
 	Therefore, $P_i$ and $P'_j$ are quasiconformally equivalent for any $i\geq m+1$ and for any $j\in \mathbb N$.
 	We may also see that the maximal dilatations of quasiconformal mappings from $P_i$ onto $P_j'$ $(i\geq m+1, j\in \mathbb N)$ can be uniformly bounded.
 	From Lemma \ref{Lemma:Gluing} we see that $E_m$ and $E'$ are quasiconformally equivalent.
 	
 	From the assumption, $R_1\setminus (S_1'\cup E_m)$ and $R_2\setminus (S_2\cup E')$ are quasiconformally equivalent.
 	By using Lemma \ref{Lemma:Gluing} again, we conclude that $R_1$ and $R_2$ are quasiconformally equivalent.
 	
 	The same argument works for $f^{-1}$ when the genus of $S_1$ is greater than the genus of $S_2$. 
 	Thus, we complete the proof of the theorem.

 \end{proof}
 
 \section{A universality of Schottky regions \\ and the universal Schottky space}

 \label{sec:universal}
 Let $G_g$ $(g>1)$ be a Schottky group of genus $g$. 
 Then, the limit set $\Lambda (G_g)$ of $G_g$ is a Cantor set in $\widehat{\mathbb C}$.
 We call the complement $\Omega (G_g)$ of $\Lambda (G_g)$, which is the region of discontinuity of $G_g$, a \emph{Schottky region} for genus $g$.
 
 Let $\Omega (G_g')$ be another Schottky region for the same genus $g$. Then the quotient surfaces $X:=\Omega (G_g)/G_g$, $X':=\Omega (G'_g)/G'_g$ are compact Riemann surfaces of genus $g$.
 We see that there is a quasiconformal mapping from $X$ onto $X'$ and the mapping is lifted to a group equivariant quasiconformal map from $\Omega (G_g)$ onto $\Omega (G'_g)$.
 Therefore, Schottky regions
 $\Omega (G_g)$ and $\Omega (G'_g)$ for genus $g$ are quasiconformal equivalent as open Riemann surfaces of infinite type.
 In fact, the quasiconformal mapping is extended to a quasiconformal mapping on $\widehat{\mathbb C}$.
 
 We also see in Example \ref{Ex:nonQC} that for a Kleinian group $G'$ of Schottky type with cusps, $\Omega (G_g)$ and $\Omega (G')$ are not quasiconformally equivalent while both are the complements of some Cantor sets.

Now, we consider a Schottky group $G_h$ of genus $h\not= g$.
Of course, there are no group equivariant quasiconformal mappings between $\Omega (G_g)$ and $\Omega (G_h)$ since those groups represent topologically different Riemann surfaces.
However, it may be possible that $\Omega (G_g)$ and $\Omega (G_h)$ are quasiconformally equivalent as open Riemann surfaces.
In fact, it is always possible.
We may show the following:
\begin{thm}
\label{thm:universal}
Schottky regions are quasiconformally equivalent to each other.
More precisely, for any Schottky groups $G, G'$ there exists a quasiconformal mapping $f$ on $\widehat{\mathbb C}$ such that $f(\Omega (G))=\Omega (G')$.
\end{thm}

As an immediate consequence, we have the following universality of Teichm\"uller spaces of Schottky regions.
\begin{Cor}
	For any $g, h >1$, the Teichm\"uller space of a Schottky region of genus $g$ and the Teichm\"uller space of a Schottky region of genus $h$ are the same.
\end{Cor}

\noindent
{\it Proof of Theorem \ref{thm:universal}.}
Let $P$ be a pair of pants bounded by three hyperbolic geodesics $\alpha_1, \alpha_2, \alpha_3$ of length one.
We make infinite copies $\{P_n\}_{n\in \mathbb Z}$ of $P$ and construct a Riemann surface $X_{\infty}$ as follows (see also \textsc{Figure} \ref{Fig1}).

\begin{figure}\centering
\includegraphics[width=14cm]{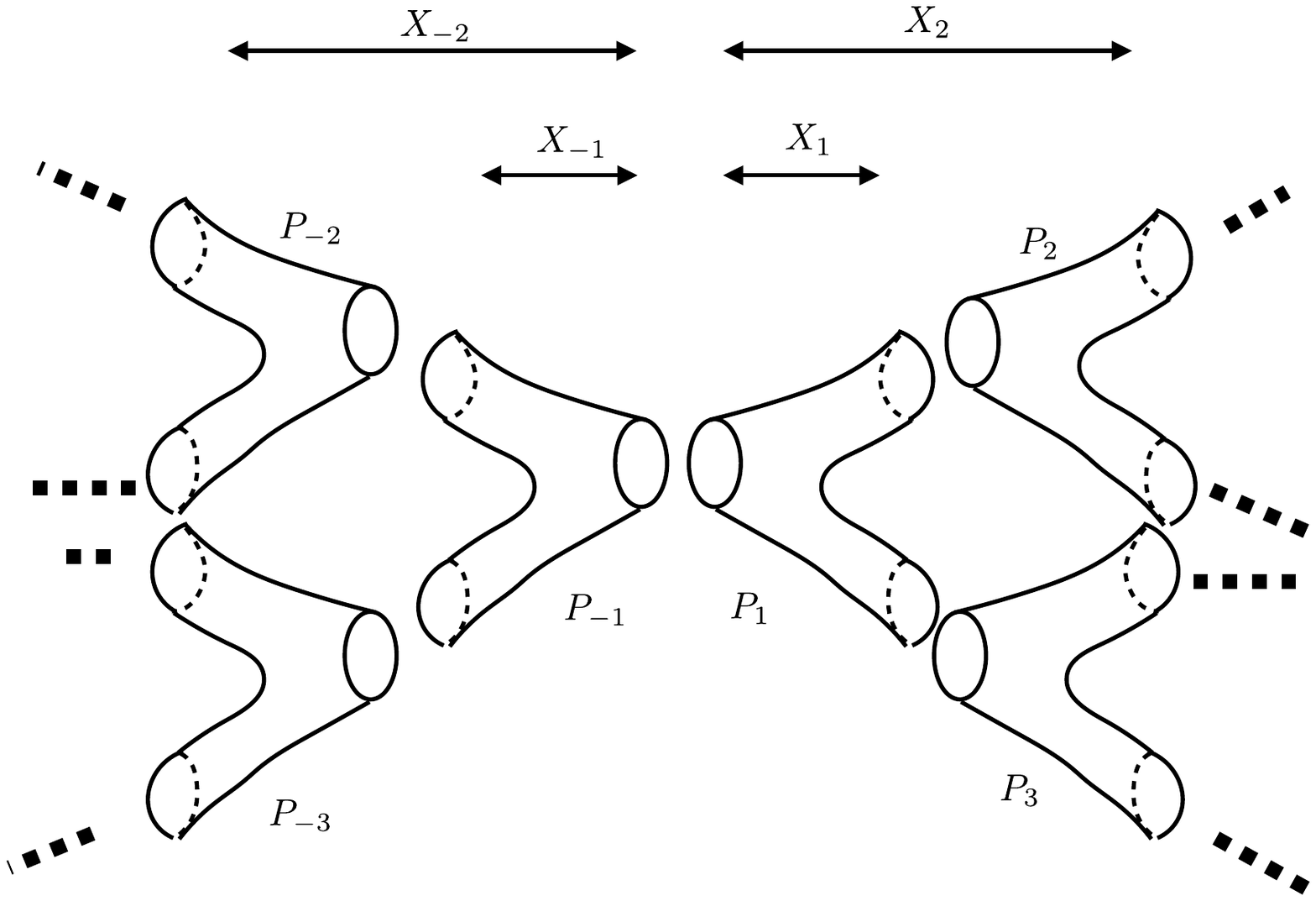}
\caption{}
\label{Fig1}
\end{figure}

Let $\alpha_{1, n}, \alpha_{2, n}$ and $\alpha_{3, n}$ be boundary curves of $P_n$ corresponding to $\alpha_1, \alpha_2$ and $\alpha_3$, respectively.
First, we put $X_1=P_1$, which is the surface of the 1st generation.
We glue $P_1$ and $P_2$ by identifying $\alpha_{2, 1}$ and $\alpha_{1, 2}$.
We also glue $P_1$ and $P_3$ by identifying $\alpha_{3, 1}$ and $\alpha_{1, 3}$.
The resulting surface denoted by $X_2$ is the surface of the 2nd generation, which is bounded by $5$ geodesics, $\alpha_{1, 1}, \alpha_{2, 2}, \alpha_{3, 2}, \alpha_{2, 3}$ and $\alpha_{3, 3}$.
Inductively, we make $X_{k+1}$ from $X_k$ $(k\in \mathbb N)$ by attaching copies of $P$ along all boundary curves of $X_k$ except $\alpha_{1. 1}$.
Symmetrically, we make $X_{-k}$ for $k\in \mathbb N$ (see \textsc{Figure} \ref{Fig1}).

We obtain the Riemann surface $X_{\infty}$ by identifying $\alpha_{1, 1}\subset\partial\cup_{k\in \mathbb N}X_k$ and $\alpha_{1, -1}\subset\partial\cup_{k\in \mathbb N}X_{-k}$.
Then, both $X_k$ and $X_{-k}$ are subsurfaces of $X_{\infty}$ bounded by $2^k+1$ geodesics of length one.
$X_k$ is made by $P_1, P_2, \dots , P_{2^k -1}$ and $X_{-k}$ is by $P_{-1}, \dots ,  P_{-2^k+1}$.

Let $G$ be a Schottky group of genus $g>1$.
We show that $\Omega(G)$ is quasiconformally equivalent to $X_{\infty}$.

From the definition of Schottky groups, there are mutually disjoint $2g$ Jordan curves $C_1, C_2, \dots , C_{2g}$ in $\widehat{\mathbb C}$ such that the outside of them, which is denoted by $F_g$, is a fundamental domain for $G$. 
The group $G$ is a free group of rank $g$ generated by $\gamma_1, \dots , \gamma_g$ and each $\gamma_j$ maps the inside of $C_{2j-1}$ onto the outside of $C_{2j}$ $(j=1, \dots , g)$.
Thus, $\Omega (G)$ is constructed from infinite copies of $F_g$ by gluing their boundary curves according to those correspondences (see 
\textsc{Figure} \ref{Fig2} for $g=3$). 
The correspondence gives a regular exhaustion of $\Omega (G)$
\begin{equation*}
	F_g=W_0\subset W_1\subset \dots \subset W_n\subset W_{n+1}\subset \dots \subset \Omega (G).
\end{equation*}
 The precise construction is the following.
 
 We start at $W_0:=F_g$. It is a region bounded by $2g$ simple closed curves $C_1, C_2, \dots , C_{2g}$.
 We put
 \begin{equation*}
 	W_1=\mathrm{Int}\left (\overline{W_0}\cup \bigcup_{j=1}^{g}\gamma_{j}^{\pm 1}(\overline{F_g})\right ).
 \end{equation*}
 $W_1$ is a region bounded by $2g(2g-1)$ simple closed curves.
 
 Inductively, we make
 \begin{equation*}
 	W_n :=\mathrm{Int}\left (\overline{W_{n-1}}\cup\bigcup_{\gamma\in S_n}\gamma (\overline{F_g})\right ),
 \end{equation*}
 where $S_n\subset G$ is the set of $\gamma\in G$ whose word lengths with respect to $\gamma_1^{\pm 1}, \dots , \gamma_g^{\pm 1}$ are precisely $n$.
 For each component $c$ of $\partial W_{n-1}$, there exist a unique $\gamma\in S_n$ and a unique $C\in \{C_1, \dots , C_{2g}\}$ such that $c=\gamma (C)$.
 Thus, $W_n$ is a region bounded by $2g(2g-1)^n$ simple closed curves coming from $C_1, \dots C_{2g}$.
 We also see that the region $W_n$ consists of $N(g):=1+\sum_{k=0}^{n-1}2g(2g-1)^k$ copies of $F_g$.
 
 \medskip
 
 Next, we make a regular exhaustion of $X_{\infty}$ to give a quasiconformal mapping from $X_{\infty}$ onto $\Omega (G)$.
 
 Let $k\in \mathbb N\cup\{0\}$ with $2^k< 2g-1<2^{k+1}$.
 In the above construction of $X_{\infty}$, we consider a subsurface of $X_{\infty}$ made by $P_1, \dots , P_{2g-1}$ and denote it by $\mathcal F_g$.
 We see that $X_k\subset \mathcal F_g\subset X_{k+1}$ and $\mathcal F_g$ is bounded by $2g$ closed geodesics.
Since both $\mathcal F_g$ and $F_g$ are Riemann surfaces of genus zero bounded by $2g$ simple closed curves, there exists a quasiconformal mapping $F$ from $\mathcal F_g$ onto $F_g$.
The quasiconformal mapping $F$ yields a correspondence between the set of boundary curves of $\mathcal F_g$ and that of $F_g$.
We put $\mathcal C_j=F^{-1}(C_j)$ $(j=1, \dots , 2g)$.

By using this correspondence between $C_j$ and $\mathcal C_j$ together with the configuration of $\{W_n\}_{n=0}^{\infty}$ by copies of $F_g$, we construct a regular exhaustion of $X_{\infty}$, 
 \begin{equation*}
 	\mathcal F_g=\widetilde{W}_0\subset\widetilde{W}_1\dots \subset \widetilde{W}_n\subset \widetilde{W}_{n+1}\subset \dots \subset X_{\infty}.
 \end{equation*}

\begin{figure}
\centering
\includegraphics[width=14cm]{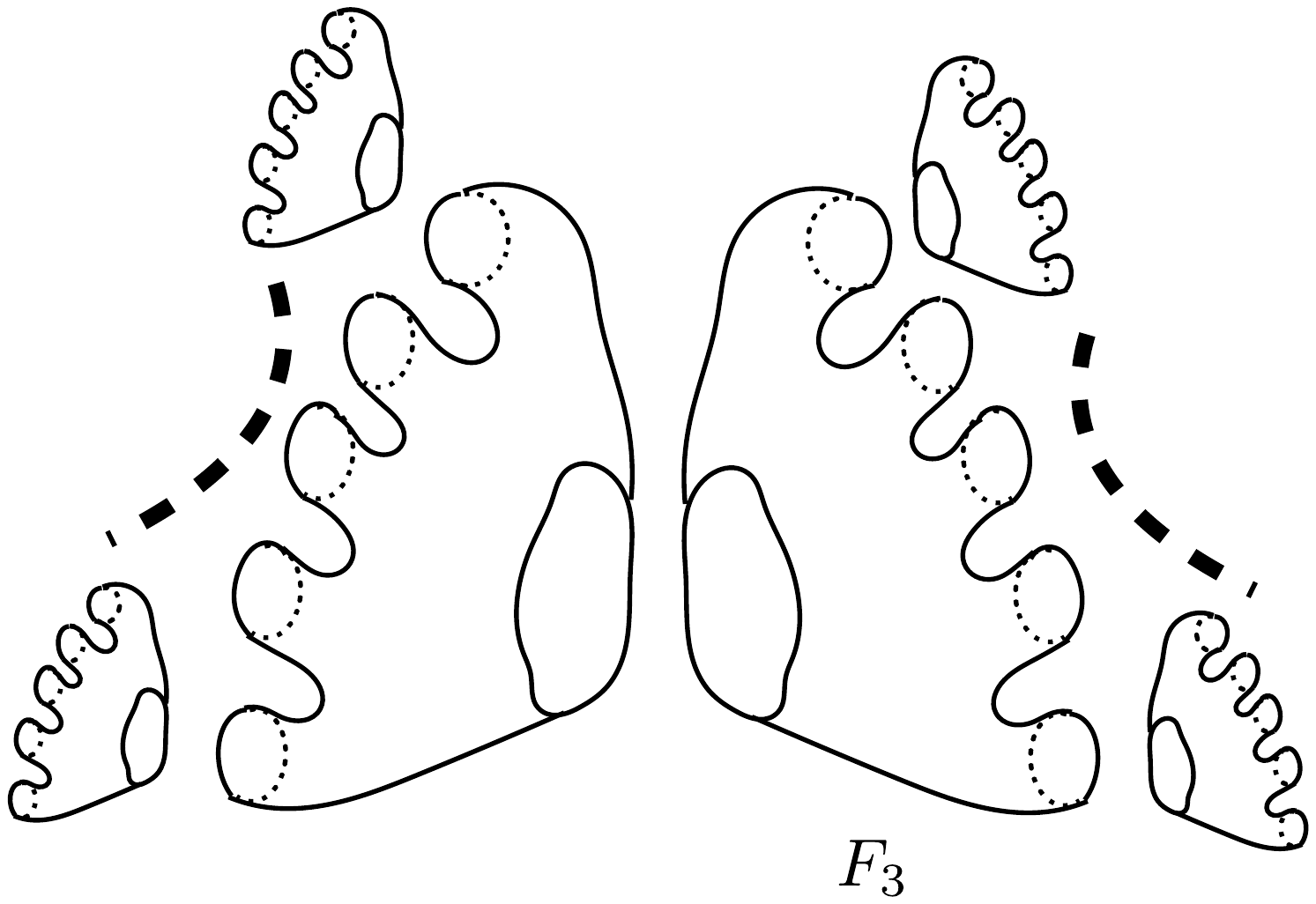}
\caption{}
\label{Fig2}
\end{figure}

Because of those constructions of the exhaustions, the quasiconformal mapping $F : \mathcal F_g\to F_g$ gives a quasiconformal mapping $\widetilde{F}$ from $X_{\infty}\setminus\cup_{n\in \mathbb N}\partial \widetilde{W_n}$ onto $\Omega (G)\setminus\cup_{n\in \mathbb N}\partial W_n$.
Noting that there are finitely many boundary behaviors of $\widetilde{F}$ near $\cup_{n\in \mathbb N}\partial \widetilde{W_n}$, we see that from Lemma \ref{Lemma:Gluing} that $\Omega (G)$ and $X_{\infty}$ are quasiconformally equivalent.

Let $G'$ be another Schottky group. Using the same argument as above for $G$, we may show that $\Omega (G')$ is quasiconformally equivalent to $X_{\infty}$. Hence, we conclude that $\Omega (G)$ and $\Omega (G')$ are quasiconformally equivalent.
As we have already noted (cf. \cite{ShigaKlein}), every quasiconformal mapping on $\Omega (G)$ is extended to a quasiconformal mapping on $\widehat{\mathbb C}$.
Thus, we have a quasiconformal mapping $f$ on $\widehat{\mathbb C}$ with $f(\Omega (G))=\Omega (G')$ as desired.
\qed

\medskip

\noindent
{\bf The universal Schottky space}

   Let $\mathcal{C}$ be the standard middle $\frac{1}{3}$-Cantor set for $[-1, 1]$. 
   It is obtained by removing the middle one thirds open intervals from $[-1, 1]$ successively. Let us recall the construction.
   
   First, we remove an open interval $J_1$ of length $2/3$ from $E_0:=I=[-1, 1]$ so that $I\setminus J_1$ consists of two closed intervals $I_1^{-1}, I_1^{1}$ of the same length, where $I_1^{-1}\subset \mathbb R_{<0}$ and $I_1^{1}\subset \mathbb R_{>0}$.
We put $E_1=I_1^{-1}\cup I_1^{1}$.
We remove an open interval of length $\frac{1}{3} |I_1^{i}|$ from each $I_1^{\pm 1}$ so that the remainder $E_2$ consists of four closed intervals of the same length, where $|J|$ is the length of an interval $J$.
Inductively, we define $E_{k+1}$ from $E_k=\bigcup_{i=-2^{k-1}}^{-1}I_k^{i}\cup \bigcup_{i=1}^{2^{k-1}}I_k^i$ by removing an open interval of length $\frac{1}{3}|I_k^{i}|$ from each closed interval $I_k^{i}$ of $E_k$ so that $E_{k+1}$ consists of $2^{k+1}$ closed intervals of the same length.
The Cantor set $\mathcal C$ is defined by
\begin{equation*}
	\mathcal C =\cap_{k=1}^{\infty}E_k.
\end{equation*}

   We put $\widehat{X}:=\widehat{\mathbb C}\setminus {\mathcal C}$.
      We denote the Teichm\"uller space $\mathscr{T}(\mathcal C)$ of $\mathcal{C}$ by $\mathscr{S}$. Then, we insist the following:
   \begin{thm}
   \label{thm:universalS}
   	For any $g>1$, there exists a holomorphic injection 
   	\begin{equation}
   		\iota_{g} : \mathscr{S}_g\hookrightarrow \mathscr{S}
   	\end{equation}
   	similar to (\ref{eqn:UniversalT}).
   \end{thm}
   \begin{proof}
   We take a pants decomposition $\{\mathcal P_n\}_{n\in \mathbb Z}$ of the Riemann surface $\widehat X:= \widehat{\mathbb C}\setminus {\mathcal C}$ as follows.
   
 We denote the imaginary axis by $C_0^0$.
 For any $(k, i)$ $(k\in \mathbb Z\setminus\{0\}; i=\pm 1, \dots \pm 2^{k-1})$, we take a circle $C_k^{i-1}$ which is a circle centered at the midpoint of $I_k^i$ with radius $\frac{5}{6}|I_k^i|$.
 We see that all $C_k^i$'s are mutually disjoint curves in $\widehat X$ and each $C_k^{i}$ contains $C_{k+1}^{\varepsilon (i)(2|i|-1)}, C_{k+1}^{2i}$, where $\varepsilon(i)=-1$ if $i<0$ and $\varepsilon(i)=1$ if $i>0$.
  Hence, they make a pants decomposition of $\widehat X$.
 A pair of pants bounded by $C_0^0$, $C_1^1$ (resp. $C_1^{-1}$) and $C_1^2$ (resp. $C_1^{-2}$) is denoted by $\mathcal P_1$ (resp. $\mathcal P_{-1}$). 
 We also denote by $\mathcal P_{\varepsilon(i)(2^{k}+(i-1))}$ a pair of pants bounded by $C_k^i$, $C_{k+1}^{\varepsilon (i)(2|i|-1)}$ and $C_{k+1}^{2i}$.
 Obviously, for every $n$ with $|n|\geq 2$, $\mathcal P_n$ is conformally equivalent to $\mathcal P_2$.
 
 Because of the construction of $\{\mathcal P_n\}_{n\in \mathbb Z}$, the configuration of the pants decomposition $\{\mathcal P_n\}_{n\in \mathbb Z}$ of $\widehat X$ is exactly the same as that of the Riemann surface $X_{\infty}$ of the proof of Theorem \ref{thm:universal}.
 It is also seen that each $\mathcal P_n$ is quasiconformally equivalent to $P_n$.
From Lemma \ref{Lemma:Gluing}, we see that the Riemann surface $X_{\infty}$  is quasiconformally equivalent to $\widehat X$.
   
   	Let $G_g$ be a Schottky group of genus $g$ and $\Omega (G_g)$ the region of discontinuity of $G_g$. From Theorem \ref{thm:universal} and the above argument, we see that there exists a quasiconformal mapping $f : \widehat{\mathbb{C}} \to \widehat{\mathbb{C}}$ with $f(\widehat{X})= \Omega(G_g)$.
   	For each quasiconformal deformation $h$ of $G_g$, $H_h:=h\circ f$ is a quasiconformal deformation of the Riemann surface $\widehat{X}$.
   	It is obvious that $h_1$ and $h_2$ are equivalent as quasiconformal deformations of $G_g$ if and only if $H_{h_1}$ and $H_{h_2}$ are equivalent as quasiconformal deformations of $\widehat{X}$.
   	Thus, we have a well-defined map $\iota_g : \mathscr{S}_g \to \mathscr{S}$.
   	The injectivity of the map follows from the definitions of $\mathscr{S}_g$ and $\mathscr{S}$.
   	
   	The complex structure of $\mathscr{S}_g$ is defined by that of the space of Beltrami differentials.
   	It is the same for the complex structure of $\mathscr{S}$.
   	Hence, the map $\iota_g$ is holomorphic. 
   \end{proof}

\end{document}